\documentclass[11pt]{amsart}

\usepackage{graphicx}
\usepackage{amssymb}
\usepackage{epstopdf}
\usepackage{amsmath,amsthm,amscd,amssymb, mathrsfs}

\usepackage[colorlinks,citecolor=red,pagebackref,hypertexnames=false]{hyperref}

\numberwithin{equation}{section}
\numberwithin{figure}{section}

\theoremstyle{plain}
\newtheorem{theorem}{Theorem}[section]
\newtheorem{lemma}[theorem]{Lemma}
\newtheorem{corollary}[theorem]{Corollary}
\newtheorem{proposition}[theorem]{Proposition}
\newtheorem{conjecture}[theorem]{Conjecture}

\theoremstyle{definition}
\newtheorem{definition}[theorem]{Definition}

\newtheorem{problem}[theorem]{Problem}

\theoremstyle{remark}

\newtheorem{case[theorem]}{Case}

\title[The $k$-resultant modulus set problem]{ The $k$-resultant modulus set problem on  algebraic varieties over finite fields }

\author{David Covert, Doowon Koh, and Youngjin Pi}

\thanks{The second listed author was supported by the Basic Science Research Program through the National
Research Foundation of Korea funded by the Ministry of Education, Science and Technology
 (NRF-2015R1A1A1A05001374).}


\begin{document}

\begin{abstract} We study the $k$-resultant modulus set problem in the $d$-dimensional vector space $\mathbb F_q^d$ over the finite field $\mathbb F_q$ with $q$ elements. Given $E\subset \mathbb F_q^d$ and an integer $k\ge 2$, the $k$-resultant modulus set, denoted by $\Delta_k(E)$, is defined as
$$ \Delta_k(E)=\{\|x^1\pm x^2 \pm \cdots \pm x^k\|\in \mathbb F_q: x^j\in E, ~j=1,2,\ldots, k\},$$
where $\|\alpha\|=\alpha_1^2+\cdots+ \alpha_d^2$ for $\alpha=(\alpha_1, \ldots, \alpha_d) \in \mathbb F_q^d.$
In this setting, the $k$-resultant modulus set problem is to determine the minimal cardinality of $E\subset \mathbb F_q^d$ such that $\Delta_k(E) = \mathbb F_q$ or $\mathbb{F}_q^*$.  This problem is an extension of the Erd\H{o}s-Falconer distance problem.  In particular, we investigate the $k$-resultant modulus set problem with the restriction that the set $E\subset \mathbb F_q^d$  is contained in a specific algebraic variety.  Energy estimates play a crucial role in our proof. 
\end{abstract}

\maketitle

\section{Introduction}
Let $g_d(N)$ denote the minimal number of distinct distances between $N$ distinct points in Euclidean space $\mathbb{R}^d$.  In 1946, Paul Erd\H os (\cite{Er46}) conjectured that $g_2(N) \gtrsim N/\sqrt{\log N}$, and $g_d(N) \gtrsim N^{2/d}$ for $d \geq 3$.  Here and throughout, we use the notation $X \gtrsim Y$ if there exists a positive constant $c$ such that $X \geq cY$.  Furthermore, we use $X \approx Y$ if $X \gtrsim Y$ and $Y \gtrsim X.$
 Erd\H os' conjecture arose from considering $N$ points arranged on regular polygons and subsets of the integer lattice $\mathbb{Z}^d$.  Guth and Katz (\cite{GuKa10}) recently established the sharp exponent for dimension two. More precisely, they proved that $g_2(N) \gtrsim N/\log{N}.$  In higher dimensions, Solymosi and Vu (\cite{SV08}) obtained the best known bound $g_d(N) \gtrsim N^{2/d - 2/(d^2+2d)}$ for $d \ge 3$, results which are far from the conjecture. \\

For $E \subset \mathbb{F}_q^d$, we define the distance set of $E$ to be 
$$\Delta(E) = \{\| x - y\|\in \mathbb F_q : x ,y \in E\},$$
where $\| (x_1, \dots , x_d) \| = x_1^2 + \dots + x_d^2$.  Bourgain, Katz, and Tao were the first to consider such an analog of the distance problem.  They proved (\cite{BKT04}) that if $E \subset \mathbb{F}_p^2$, where $p \equiv 3 \pmod{4}$ is prime, and given $\delta > 0$ such that $q^{\delta} \lesssim |E| \lesssim q^{2 - \delta}$, then there exists $\epsilon = \epsilon(\delta)$ such that $|\Delta(E)| \geq |E|^{1/2 + \epsilon}$.  Their proof involved finding a relationship between incidence geometry in $\mathbb{F}_p^2$ and the distance set.  Unfortunately, the relationship between $\delta$ and $\epsilon$ is difficult to deduce from their proof.  Additionally, this formulation of the finite field distance problem leaves open a few degenerate possibilities.  For example, if $p \equiv 1 \pmod{4}$, then there is an element $i \in \mathbb{F}_p$ such that $i^2 = -1$, and if $d = 2k$ is even, then we could consider $E = \{(x_1, ix_1, \dots , x_k,ix_k) : x_i \in \mathbb{F}_p \} \subset \mathbb{F}_p^{d}$.  In this case, $|E| = p^{d/2}$, and yet $\Delta(E) = \{0\}$.  Similar examples could be considered when the dimension $d$ is odd.  Furthermore, note that if $E = \mathbb{F}_p^2$, then $|\Delta(E)| = p = |E|^{1/2}$ with no exponential gain.\\

In order to circumvent these degenerate cases, it is helpful to recast the distance set problem, and we use the Falconer distance problem as our motivation.   For $E \subset \mathbb{R}^d$, define $\Delta(E) = \{|x - y| : x ,y \in E\} \subset \mathbb{R}$.  Falconer showed (\cite{Fa85}) that if $E \subset [0,1]^d$ is compact and has Hausdorff dimension $\dim(E) > \frac{d+1}{2}$, then the distance set $\Delta(E)$ has positive one-dimensional Lebesgue measure.  He also constructed a set with Hausdorff dimension $d/2$ that had measure zero.  This led him to the conjecture that if a compact set $E \subset [0,1]^d$ has Hausdorff dimension $\dim(E) > d/2$, then the associated distance set must have positive Lebesgue measure.  See \cite{Er05, Wo99} for the most recent progress on the problem.  In this spirit, Iosevich and Rudnev (\cite{IR07}) restated the finite field distance question as follows.
\begin{problem}\label{EFDP}
Let $E \subset \mathbb{F}_q^d$.  Find the minimal exponent $\alpha$ such that if $|E| \geq Cq^{\alpha}$ for a sufficiently large constant $C$, then we have $\Delta(E) = \mathbb{F}_q$.
\end{problem}
Also, we could consider a weakened version of the problem.

\begin{problem}\label{weakEFDP}
Let $E \subset \mathbb{F}_q^d$.  Find the minimal exponent $\beta$ such that if $|E| \geq C q^{\beta}$, then there exists a constant $0 < c \leq 1$ such that $|\Delta(E)| \geq cq$.
\end{problem}
Collectively, we refer to Problems \ref{EFDP} and \ref{weakEFDP} as the Erd\H os-Falconer Distance Problem.  We will write $\alpha_2(d)$ to denote the smallest exponent $\alpha$ that solves Problem \ref{EFDP}, and we write $\beta_2(d)$ to denote the smallest exponent $\beta$ that solves Problem \ref{weakEFDP}, so that $\beta_2(d) \leq \alpha_2(d)$.  Note that if $q = p^2$, then $\mathbb{F}_q$ contains a subfield isomorphic to $\mathbb{F}_p$, and hence we can take $E$ isomorphic to $\mathbb{F}_p^d \subset \mathbb{F}_q^d$ yielding a set such that $|E| = q^{d/2}$ and still $|\Delta(E)| = \sqrt{q}$.  Likewise, if $-1$ is a square and the dimension $d$ is even, we could take the same degenerate example as before: $E = \{(x_1, ix_1, \dots , x_{d/2}, i x_{d/2}) : x_i \in \mathbb{F}_q\}$ yields $|\Delta(E)| = 1$.  Thus, we have $\alpha_2(d) \geq d/2$.  Iosevich and Rudnev (\cite{IR07}) showed that we have $\alpha_2(d) \leq \frac{d+1}{2}$.  At first blush it is plausible that $\alpha_2(d) = d/2$ in all dimensions, in line with Falconer's conjecture.  However, the authors in \cite{HIKR10} showed that $\beta_2(d) = (d+1)/2$, at least for odd dimensions $d \geq 3$.  It is still possible that $\alpha_2(d) = d/2$ when $d$ is even, but there has been no further progress in this direction.  The only progress that has been made (\cite{BHIPR14, CEHIK09}) is the bound $\beta_2(2) \leq 4/3$.  Further progress has proved very difficult, indeed.\\

Rather than dealing with the distance $\| x - y\| \in \mathbb{F}_q$, the authors (\cite{CKP14}) studied  the quantity $\| x^1 \pm \dots \pm x^k \|\in \mathbb{F}_q$ for $x^i\in \mathbb F_q^d$ and an integer $k\ge 2.$  Since our results are independent of the sign $\pm$, we simply define
\begin{equation}\label{eq1.1} \Delta_k(E) = \{ \| x^1 -x^2- \cdots - x^k \| : x^i \in E\} \subset \mathbb{F}_q, \end{equation}
and we call this the $k$-resultant set of $E$.  We may think of $\Delta_2(E)$ as the distance set, and then $\Delta_k(E)$ is a generalization of the distance set.  The questions we ask regarding the distribution of $\Delta_k(E)$ are similar to  those asked in the Erd\H os-Falconer Distance Problem.
\begin{problem}
Let $E \subset \mathbb{F}_q^d$.  Find the minimal exponent $\alpha$ such that if $|E| \geq Cq^{\alpha}$ for a sufficiently large constant $C$, then we have $\Delta_k(E) = \mathbb{F}_q$.  Find the minimal exponent $\beta$ such that if $|E| \geq Cq^{\beta}$, then there exists a constant $0 < c \leq 1$ such that $|\Delta_k(E)| \geq cq$.
\end{problem}
In line with our earlier notation, we define $\alpha_k(d)$ to be the smallest exponent $\alpha$ such that $|E| \geq Cq^{\alpha}$ implies $\Delta_k(E) = \mathbb{F}_q$, and we put $\beta_k(d)$ to be the smallest exponent $\beta$ such that $|E| \geq Cq^\beta$ implies $|\Delta_k(E)| \geq cq$ for some constant $0  < c \leq 1$.  We first show that $\alpha_k(d) = \frac{d+1}{2}$ when the dimension $d$ is odd for general $q$.

\begin{theorem} \label{sharpodd}  Suppose that $d\ge 3$ is odd and $-1\in \mathbb F_q$ is a square. Then we have
\[ \alpha_k(d) = \frac{d+1}{2} \quad \mbox{for all integers}~~k\ge 2.\]
\end{theorem}
\begin{proof} 
In \cite{IR07}, Iosevich and Rudnev proved that $ \alpha_2(d) \le (d+1)/2$.  We note that for integers $k_1$ and $k_2$ such that $2 \leq k_1 \leq k_2$, we have $\alpha_{k_2}(d) \leq \alpha_{k_1}(d)$ and $\beta_{k_2}(d) \leq \beta_{k_1}(d)$.  Hence, $\alpha_k(d) \le \alpha_2(d)\le (d+1)/2$ for $k\ge 2.$
Therefore, assuming that $-1\in \mathbb F_q$ is a square, it suffices to prove that $ \alpha_k(d) \ge (d+1)/2$ for all integers $k\ge 2$ and odd $d\ge 3.$
Now,  if $d = 2n + 1\ge 3$ is odd, we can take
\begin{equation}\label{E}
E = \{(t_1, it_1, \dots , t_n, it_n, s) : t_i, s \in \mathbb{F}_q\}.
\end{equation}
Here, $|E| = q^{\frac{d+1}{2}}$, and yet
\[
\Delta_k(E) = \{\sigma^2 : \sigma \in \mathbb{F}_q\} \quad \mbox{for all integers}~~k\ge 2.
\]
Since there are only $\frac{q+1}{2} < q$ squares in $\mathbb{F}_q$, the result follows.\end{proof}

Thus we have shown that $\alpha_k(d) = (d+1)/2$ is sharp in odd dimensions $d \geq 3$.  On the other hand, if $d\ge 2$ is even, then sets like $E$ in \eqref{E} may not be constructed. Alternatively, if $-1$ is a square in $\mathbb{F}_q$ and $d\ge 2$ is even, then we may take the set $E= \{(t_1, i t_1, \dots , t_{d/2}, it_{d/2} : t_i \in \mathbb{F}_q\} \subset \mathbb F_q^d$ as before. In this case, we see that  $\beta_k(d) \ge d/2$ for all integers $k\ge 2$ and even $d\ge 2$, because $|E|=q^{d/2}$ and $|\Delta_k(E)|=|\{0\}|=1$ for all $k\ge 2.$ Likewise, it was shown in \cite{CEHIK09} that $\beta_2(d) = \frac{d+1}{2}$.  In view of these examples, the following conjecture was given by the authors in \cite{CKP14}. 
 
\begin{conjecture} If $d\ge 2$ is even, then 
$$ \beta_k(d)=\frac{d}{2} \quad \mbox{for all integers} ~~k\ge 2.$$
\end{conjecture} 

In addition, they obtained  the following result.
\begin{proposition}Let $E\subset \mathbb F_q^d.$ Suppose that $C$ is a sufficiently large constant. Then the following results hold:
\begin{enumerate}
\item If  $d=4$ or $6,$ and
 $|E|\geq C q^{\frac{d+1}{2}-\frac{1}{6d+2}},$   then $|\Delta_3(E)|\geq c q$ for some $0 < c \leq 1.$
\item If $d\ge 8$ is even and $|E|\geq C q^{\frac{d+1}{2}-\frac{1}{6d+2}},$ then $|\Delta_4(E)|\geq cq$ for some $0 < c \leq 1.$
\item Suppose that $d\geq 8$ is even.
Then given $\varepsilon>0,$ there exists $C_{\varepsilon}>0$ such that
if $|E|\ge C_{\varepsilon}  q^{\frac{d+1}{2} - \frac{1}{9d -18} + \varepsilon},$ then $|\Delta_3(E)|\geq cq$ for some $0 < c \leq 1.$
\end{enumerate}
In other words, we have $\beta_3(d) \leq \frac{d+1}{2} - \frac{1}{6d + 2}$ for $d = 4, 6$, and $\beta_4(d) \leq \frac{d+1}{2} - \frac{1}{6d + 2}$ for even values $d \geq 8$.
In addition, we have $\beta_3(d) \leq \frac{d+1}{2} - \frac{1}{9d -18} $ for even values $d \geq 8.$

\end{proposition}




The results of the above proposition are very interesting in that the exponent $(d+1)/2$ which is sharp in odd dimensions for the $k$-resultant modulus set problem can be improved in even dimensions.

\subsection{Purpose of this paper} 
In this paper,  we investigate the minimal cardinality of sets $E$ lying on an algebraic variety $V$ of $\mathbb F_q^d$ such that  $\Delta_k(E)\supset \mathbb F_q$ or $\mathbb F_q^*.$ In the specific case when $k=2$ and the set $E$ is contained in a unit sphere $S_1=\{x\in \mathbb F_q^d: \|x\|=1\}$,  the authors in \cite{HIKR10} proved the following results.

\begin{proposition} \label{sphericalerdos} Let $E \subset {\mathbb
F}_q^d$, $d \ge 3$, be a subset of the sphere $S_1=\{x \in {\mathbb
F}_q^d:\, \|x\|=1\}$.

\begin{enumerate} \item If $|E| \ge Cq^{\frac{d}{2}}$ with a
sufficiently large constant $C$, then there exists $c>0$ such that
 $ |\Delta_2(E)| \ge cq.$

\item If $d$ is even, then under the same assumptions as above,
$ \Delta_2(E)={\mathbb F}_q.$

\item If $d$ is even, there exists $c>0$ and $E \subset S_1$ such that
$|E| \ge cq^{\frac{d}{2}} $
and
 $\Delta_2(E) \not={\mathbb F}_q. $

\item If $d$ is odd and $|E| \ge Cq^{\frac{d+1}{2}}$ with a sufficiently large constant $C>0$, then $ \Delta_2(E)={\mathbb F}_q.$ 

\item If $d$ is odd, there exists $c>0$ and $E \subset S_1$ such that
$ |E| \ge cq^{\frac{d+1}{2}} $
and
 $\Delta_2(E) \not={\mathbb F}_q. $
 
\end{enumerate}
\end{proposition}

The main goal of this paper is to address the $k$-resultant modulus set problem in the case when  a set $E$ lies on any algebraic variety with the same Fourier decay as the sphere $S_1\subset \mathbb F_q^d.$    As a consequence, we shall see that if $k$ becomes larger, the exponent $d/2$ can be significantly improved in all dimensions $d \geq 2$.

\section{Discrete Fourier analysis and statement of main results} \label{sec2}

Before we state our main results, we review some background in Fourier analysis over finite fields.
Given a function $f:\mathbb F_q^d \to \mathbb C$,  the Fourier transform of $f$, denoted by $\widehat{f},$ is defined as
$$ \widehat{f}(m) =q^{-d} \sum_{x\in \mathbb F_q^d} \chi(-m\cdot x)~f(x) \quad \mbox{for}~~m\in \mathbb F_q^d,$$
where $\chi$ denotes a nontrivial additive character of $\mathbb F_q.$ The orthogonality relation of $\chi$ yields
$$ \sum_{x\in \mathbb F_q^d} \chi(m\cdot x) =\left\{ \begin{array}{ll}  0 \quad&\mbox{if}~~m\ne(0,\ldots,0)\\
                                                                                                   q^d\quad&\mbox{if}~~m=(0,\ldots,0).\end{array} \right.$$
 Also recall that the Fourier inversion theorem yields
$$ f(x)= \sum_{m\in \mathbb F_q^d} \chi(m\cdot x)~\widehat{f}(m)$$
and the Plancherel theorem yields
$$ \sum_{m\in \mathbb F_q^d} |\widehat{f}(m)|^2 =q^{-d} \sum_{x\in \mathbb F_q^d} |f(x)|^2.$$
Throughout the paper, we shall identify a set $E\subset \mathbb F_q^d$ with the characteristic function $\chi_E$ on the set $E.$ For instance, if $E\subset \mathbb F_q^d$, then we shall denote by $\widehat{E}$ the Fourier transform of the characteristic function on $E.$

\begin{definition} Let $Q \in \mathbb F_q[x_1, \ldots, x_d]$ be a polynomial.
The variety $V:=\{x\in \mathbb F_q^d: Q(x)=0\}$ is called a regular variety  if 
$|V|\approx q^{d-1}$ and $|\widehat{V}(m)| \lesssim q^{-(d+1)/2}$ for all $m\ne (0,\ldots,0).$  In particular the regular variety $V$ is called a nondegenerate regular curve if $V\subset \mathbb F_q^2$ and $Q(x)$ does not contain any linear factor.
\end{definition}

For $j\in \mathbb F_q$,  a sphere $S_j $ is defined as 
$$ S_j=\{x\in \mathbb F_q^d: \|x\|=j \}.$$
We define a paraboloid $P$ as 
$$ P=\{x\in \mathbb F_q^d: x_1^2+ \cdots+x_{d-1}^2 =x_d\}.$$
Typical examples of regular varieties are the paraboloid and  the sphere with nonzero radius (see \cite{MT04} and \cite{IR07}, respectively).

\subsection{Statement of main results}
Our first result below is bound to the cardinality of the $k$-resultant modulus sets generated by subsets of a general regular variety. 
\begin{theorem} \label{T1} Suppose that  $V\subset \mathbb F_q^d$ is a regular variety, and assume that $k\ge 3$ is an integer and $E\subset V.$
Then if $|E|\ge C q^{\frac{d-1}{2} +\frac{1}{k-1}}$ for a sufficiently large constant $C>0$,  we have 
 $$\Delta_k(E) \supset \mathbb F_q^* \quad \mbox{for even}~~d\ge 2,$$
and 
$$\Delta_k(E) =\mathbb F_q \quad \mbox{for odd}~~d\ge 3.$$
\end{theorem}
Proposition \ref{sphericalerdos} indicates that in order to get  $\Delta_2(E)=\mathbb F_q,$ the sharp exponent for sets $E$ of $S_1$ 
must be $d/2$ for even $d\ge 4,$ and $(d+1)/2$ for odd $d\ge 3.$ On the other hand, Theorem \ref{T1} shows that  the exponent $d/2$ can be decreased to $(d-1)/2+ 1/(k-1)$ for $k\ge 3$ and any regular variety $V\subset \mathbb F_q^d, d\ge 2$ (if we are interested in getting $\Delta_k(E) \supset \mathbb F_q^*$ for even $d\ge 2$). The authors in  \cite{HIKR10} used the dot-product set estimates for deriving the size of $\Delta_2(E)$ for $E\subset S_1.$
More precisely, they  utilized the specific property that if $E$ is a subset of the unit sphere $S_1 \subset \mathbb F_q^d,$ then $\|x-y\|=2-2 x\cdot y$ for $x,y \in E$, and so
$$ |\Delta_2(E)|=|\Pi_2(E)|:=|\{x\cdot y\in \mathbb F_q: x,y\in E\}|.$$  Here, $x\cdot y = x_1y_1 + \dots + x_dy_d$ is the standard dot product.
More generally,  we see that if $k\ge 2$ and $E \subset S_1 \subset \mathbb F_q^d,$ then
\begin{equation}\label{dot} |\Delta_k(E)| = |\Pi_k(E)| :=\left|\left\{ \sum_{i=1}^k \sum_{j=1}^k  \delta_{i<j}~  \zeta_{i,j}~ x^i \cdot x^j:  x^{l} \in E ,  l=1,2,\ldots, k\right\}\right|,\end{equation}
where $\delta_{i<j} =1 $ if $i<j$ and $0$ otherwise, and $\zeta_{i,j} =1$ for $i=1$ and $-1$ otherwise.
However, if $k\ge 3,$ then it may not be simple to obtain a good lower bound on $|\Pi_k(E)|.$
Furthermore,  if the unit sphere $S_1$ is replaced by a general regular variety $V\subset \mathbb F_q^d,$ then the inequality in (\ref{dot}) can not be true in general.
For these reasons,  the dot-product set estimates may not be useful in deriving results on  the $k$-resultant modulus set problem for an algebraic variety.
Instead of the dot-product set estimates, we shall relate our problem to estimating  the $k$-energy (see Definition \ref{kenergy} below), which will yield Theorem \ref{T1}.\\

In dimension two, when a regular variety $V$ is nondegenerate, Theorem \ref{T1} can be improved.
Indeed, we have the following result.

\begin{theorem} \label{T2}Suppose that  $E$ is contained in a nondegenerate regular curve $V\subset \mathbb F_q^2.$   If $k\ge 4$ is an integer and  $|E|\ge C q^{\frac{1}{2}+\frac{1}{2k-4}}$ for a sufficiently large constant $C>0$, then $\mathbb F_q^* \subset \Delta_k(E).$
\end{theorem}

\section{Preliminary lemmas}

In this section,  we derive and collect useful lemmas and well known facts which are essential in proving our main results. As in the Erd\H{o}s-Falconer distance problem,  analyzing a counting function will be a key ingredient to deduce results on the $k$-resultant modulus set problem. Let us fix an integer $k\ge 2$ and $E\subset \mathbb F_q^d.$ For each $t\in \mathbb F_q,$ define a counting function $\nu_k(t)$ as
$$\nu_k(t)=|\{ (x^1, x^2, \ldots, x^k)\in E^k: \|x^1-x^2 -\cdots -x^k\|=t\}|.$$
\begin{lemma}\label{keylem} Let $E\subset \mathbb F_q^d$ and $k\ge 2$ be an integer.
Suppose that  there exists a constant $c>0$ independent of $q$, the size of the underlying finite field, such that 
\begin{equation}\label{C1} \frac{|E|^k}{q} > c q^{dk-\frac{d+1}{2}} \sum_{m\in \mathbb F_q^d} \left|\widehat{E}(m) \right|^{k}.\end{equation}
Then we have
$$ \Delta_k(E) \supset \mathbb F_q^* \quad\mbox{for even}~~ d\ge 2,$$
and 
$$ \Delta_k(E)=\mathbb F_q \quad\mbox{for odd}~~ d\ge 3.$$
\end{lemma}
\begin{proof}
Notice that if $\nu_k(t)>0$, then $t\in \Delta_k(E).$ 
It therefore suffices to prove that if $d\ge 2$ is even, then
\begin{equation}\label{con1} \nu_k(t) >0 \quad \mbox{for all}~~t\in \mathbb F_q^*,\end{equation}
and if $d\ge 3$ is odd then 
\begin{equation}\label{con2} \nu_k(t) >0 \quad \mbox{for all}~~t\in \mathbb F_q.\end{equation}

  Let us estimate $\nu_k(t).$ We have
$$\nu_k(t)=\sum_{x^1,\ldots, x^k \in E} S_t(x^1-\cdots-x^k), $$ 
where we recall that $S_t=\{x\in \mathbb F_q^d : \|x\|=t\}$ and we identify the set $S_t$ with the characteristic function $\chi_{S_t}$ on the set $S_t.$
Applying the Fourier inversion theorem to $S_t(x^1-\cdots-x^k)$ and using the definition of  the normalized Fourier transform,  we see 
$$\nu_k(t)= q^{dk} \sum_{m\in \mathbb F_q^d} \widehat{S_t}(m) \overline{\widehat{E}}(m) \left( \widehat{E}(m) \right)^{k-1}$$
$$= q^{-d} |S_t| |E|^k + q^{dk} \sum_{m\ne (0,\ldots,0)} \widehat{S_t}(m) \overline{\widehat{E}}(m) \left( \widehat{E}(m) \right)^{k-1} := M_t + R_t. $$
 
The size of $S_t\subset \mathbb F_q^d$ is approximately 
$ q^{d-1}$ unless $t=0, d=2,$ and $-1\in \mathbb F_q$ is not a square.
In fact,  the following explicit value of $|S_t|$  can be obtained  (see Theorem 6.26 and Theorem 6.27 in \cite{LN97}):
If $S_t \subset \mathbb F_q^d$ is the sphere, then we have
\begin{equation}\label{sizeS} |S_t| =\left\{ \begin{array}{ll}  q^{d-1} +v(t) q^{\frac{d-2}{2}} \eta( (-1)^{\frac{d}{2}} ) \quad &\mbox{for even}~~d\ge 2\\
                                              q^{d-1} + q^{\frac{d-1}{2}} \eta(  (-1)^{\frac{d-1}{2}} t )\quad &\mbox{for odd}~~d\ge 3, \end{array}  \right.\end{equation}
where $\eta$ denotes the quadratic character of $\mathbb F_q$, and  the integer-valued function $v$ on $\mathbb F_q$ is defined by $ v(t)=-1$ for $t\in \mathbb F_q^*$ and $v(0)=q-1.$     

The following Fourier dacay on $S_t\subset \mathbb F_q^d$ was given in Proposition 2.2 of \cite{KS12}:\\
If $m\ne (0, \ldots,0),$ then we have
\begin{equation}\label{decay} |\widehat{S_t}(m)| \le 2 q^{-\frac{d+1}{2}} \quad\mbox{unless} ~d\ge 2 ~\mbox{is even and}~ t=0.\end{equation}

From \eqref{sizeS} and \eqref{decay}, we see that 
$|S_t|= q^{d-1}(1 + \underline{o}(1))$ and $|\widehat{S_t}(m)|\le 2 q^{-(d+1)/2}$ if $d\ge 2$ is even, $t\in \mathbb F_q^*$ and $m\ne (0,\ldots,0)$ (or if $d\ge 3$ is odd, $t\in \mathbb F_q$ and $m\ne (0,\ldots,0)$).
Using these facts, we observe that if  $d\ge 2$ is even and $t\in \mathbb F_q^*$ (or if $d\ge 3$ is odd and $t\in \mathbb F_q$), then
$$ 0< M_t= \frac{|E|^k}{q}(1 +\underline{o}(1)) $$
and
$$ |R_t| \le q^{dk-\frac{d+1}{2}} \sum_{m\in \mathbb F_q^d} \left|\widehat{E}(m) \right|^{k}.$$
Then \eqref{con1} and \eqref{con2} follow immediately from the assumption \eqref{C1}.
Thus, the proof is complete.
\end{proof}

Lemma \ref{keylem} says that  results on the $k$-resultant modulus set problems can be deduced by determining the size of the set $E\subset \mathbb F_q^d$ satisfying the inequality \eqref{C1}.
To do this, it will be a key factor to obtain a good upper bound of  $\sum_{m\in \mathbb F_q^d} \left|\widehat{E}(m) \right|^{k}.$ 
Using a trivial estimate on it,  the Plancherel theorem yields that
\begin{equation}\label{trivialk}\sum_{m\in \mathbb F_q^d} \left|\widehat{E}(m) \right|^{k} \le \left|\widehat{E}(0,\ldots,0) \right|^{k-2} \sum_{m\in \mathbb F_q^d} \left|\widehat{E}(m) \right|^{2} = \left(\frac{|E|}{q^d} \right)^{k-2} \frac{|E|}{q^d}
=\frac{|E|^{k-1}}{q^{dk-d}}.\end{equation}
Combining this estimate with Lemma \ref{keylem},  we see that  if $E\subset \mathbb F_q^d$ with $|E|\ge C q^{(d+1)/2}$ for some large constant $C>1$ then the conclusions in Lemma \ref{keylem} hold for any integer $k\ge 2.$ We have also seen from Theorem \ref{sharpodd} that for every integer $k\ge 2,$ the exponent $(d+1)/2$ is in general sharp in odd dimensions $d\ge 3.$ In conclusion,  the inequality \eqref{trivialk} can not be improvable for all integers $k\ge 2$ in general odd dimensional case. Furthermore,  $(5)$ of Proposition \ref{sphericalerdos} also implies that 
even if  we restrict the sets $E$ to the subsets of the unit sphere $S_1\subset \mathbb F_q^d,$ the inequality \eqref{trivialk}  can not be improved in general for $k=2.$ 
However, we shall see that if $k>2$ and the set $E$ is contained in a regular variety, then  the inequality \eqref{trivialk}  can be significantly improved in any dimension $d\ge 2.$
From this observation, we shall derive our main results stated in Section \ref{sec2}.

\begin{definition}\label{kenergy}
For an even integer $k=2m \ge 2$ and $E\subset \mathbb F_q^d,$   the $k$-energy is defined as
$$ \Lambda_k(E)= \sum_{\substack{x^1, \ldots, x^k\in E\\
                                     x^1+\cdots+ x^{m} = x^{m+1} + \cdots+ x^k}} 1.$$
\end{definition}

An upper bound of $\sum_{m\in \mathbb F_q^d} \left|\widehat{E}(m) \right|^{k}$ can be written in terms of the $k$-energy $\Lambda_k(E).$ 
\begin{lemma} \label{Lem3.3} If $k\ge 2$ is even, and $E\subset \mathbb F_q^d$, then we have
$$ \sum_{m\in \mathbb F_q^d} \left|\widehat{E}(m)\right|^k = q^{-dk+d} \Lambda_k(E).$$
\end{lemma}
\begin{proof} Since $k\ge 2$ is an even integer,   we can write 
$$|\widehat{E}(m)|^k = \left(\widehat{E}(m)\right)^{k/2} \left(\overline{\widehat{E}}(m)\right)^{k/2}.$$
Then the statement follows from the definition of the Fourier transform and the orthogonality relation of $\mathbb F_q.$
\end{proof}

\begin{lemma}\label{Lem3.4} If $k\ge 3$ is odd and  and $E\subset  \mathbb F_q^d$, then we have
$$ \sum_{m\in \mathbb F_q^d}  \left|\widehat{E}(m)\right|^k \le q^{-dk+d} \left( \Lambda_{k-1}(E)~ \Lambda_{k+1}(E)\right)^{\frac{1}{2}}.$$
\end{lemma} 
\begin{proof} By Lemma \ref{Lem3.3}, we have
\begin{equation}\label{e1} \sum_{m\in \mathbb F_q^d} \left|\widehat{E}(m)\right|^{k-1} = q^{-d(k-1)+d} \Lambda_{k-1}(E) \end{equation}
and 
\begin{equation}\label{e2}\sum_{m\in \mathbb F_q^d} \left|\widehat{E}(m)\right|^{k+1} = q^{-d(k+1)+d} \Lambda_{k+1}(E). \end{equation}
Interpolating (\ref{e1}) and (\ref{e2}), we complete the proof.
\end{proof}

 The Lemma below follows immediately by combining Lemma \ref{Lem3.3} and Lemma \ref{Lem3.4} with Lemma \ref{keylem}.
\begin{lemma}\label{keylem1}  Let $E \subset \mathbb F_q^d.$ \\
$(1)$ If $k\ge 2$ is an even integer and $|E|^k \gtrsim q^{(d+1)/2} \Lambda_k(E)$, then 
$$ \Delta_k(E) \supset \mathbb F_q^* ~~\mbox{for even}~~ d\ge 2, ~~\mbox{and}~~\Delta_k(E)=\mathbb F_q ~~\mbox{for odd}~~ d\ge 3.$$

\noindent (2) If $k\ge 3$ is odd and  $|E|^k \gtrsim q^{(d+1)/2}  \left( \Lambda_{k-1}(E) \Lambda_{k+1}(E) \right)^{1/2}$, then 
$$ \Delta_k(E) \supset \mathbb F_q^* ~~\mbox{for even}~~ d\ge 2, ~~\mbox{and}~~\Delta_k(E)=\mathbb F_q ~~\mbox{for odd}~~ d\ge 3.$$
\end{lemma}

\section{Estimate of $k$-energy}

Notice from Lemma \ref{keylem1} that  main task  for the $k$-resultant modulus set results is to estimate upper bounds of
$ \Lambda_k(E)$ for even $k\ge 2$  and $\Lambda_{k-1}(E) \Lambda_{k+1}(E)$ for odd $k\ge 3.$
It is obvious that $\Lambda_2(E)=|E|$ for $E\subset \mathbb F_q^d$. 
For even $k\ge 4$, we simply see that $\Lambda_k(E) \le |E|^{k-1}$ for $E\subset \mathbb F_q^d.$
Note that this estimate is sharp in the case when $q=p^2$ for a prime $p$ and $E=\mathbb F_p^d.$ 
Combining the trivial estimate of $\Lambda_k(E)$ with Lemma \ref{keylem1},  we recover the $(d+1)/2$ exponent, the sharp one for general odd dimensions and $k\ge 2.$
On the other hand, if $k\ge 4$ is even and the set $E$ is contained in a regular variety $V\subset \mathbb F_q^d,$ 
we shall obtain much better upper bound of $\Lambda_k(E)$  than the trivial one. This enables us to prove our main results.
We begin by showing that if $E$ lies on a regular variety $V\subset \mathbb F_q^d$ then an upper bound of  the $k$-energy $\Lambda_k(E)$ can be written in terms of the $(k-2)$-energy.

\begin{lemma}\label{mlem} Let $V\subset \mathbb F_q^d$ be a regular variety.
Then  if $k\ge 4$ is even and $E\subset V,$   we have
$$ \Lambda_{k} (E) \lesssim q^{d-1} \Lambda_{k-2}(E) +q^{-1} |E|^{k-1}.$$
\end{lemma}
\begin{proof}  By the definition of $\Lambda_k(E)$,  it follows
$$\Lambda_{k}(E)=\sum_{x^1, \ldots, x^k\in E} \delta_0( x^1+\cdots+x^{k/2} - x^{k/2 +1} - \cdots -x^k) $$
where $\delta_0(x)=1$ if $x=(0,\ldots,0)$ and  $0$ otherwise.
Since $E\subset V$, it is clear that
$$ \Lambda_{k}(E) \le \sum_{x^1, \ldots, x^{k-1} \in E} V(x^1+\cdots+x^{k/2} - x^{k/2 +1} - \cdots -x^{k-1}).$$
Using the Fourier inversion theorem to $V(x^1+\cdots+x^{k/2} - x^{k/2 +1} - \cdots -x^{k-1})$ and the definition of the Fourier transform,  we see that
\begin{align*}  \Lambda_{k}(E) \le&~ q^{d(k-1)} \sum_{m\in \mathbb F_q^d} |\widehat{V}(m)| |\widehat{E}(m)|^{k-1}\\
=&~q^{d(k-1)} |\widehat{V}(0, \ldots,0)|~ |\widehat{E}(0,\ldots,0)|^{k-1}\\
 &+ q^{d(k-1)} \sum_{m\ne (0,\ldots,0)} 
|\widehat{V}(m)| |\widehat{E}(m)|^{k-1}\\
 :=& \mbox{I} + \mbox{II}.
\end{align*}
Since $V\subset \mathbb F_q^d$ is a regular variety, we see that  $|V|\approx q^{d-1}$ and $ |\widehat{V}(m)| \lesssim q^{-\frac{d+1}{2}} $ for all $m\neq (0,\ldots,0)$.  It therefore follows that
$$ \mbox{I}\sim \frac{|E|^{k-1}}{q}$$
and 
$$ \mbox{II} \lesssim q^{d(k-1) - (d+1)/2} \sum_{m\in \mathbb F_q^d} |\widehat{E}(m)|^{k-1}.$$

Since $k\ge 4$ is even, $k-1\ge 3$ is odd. By Lemma \ref{Lem3.4}, 
$$ \mbox{II}\lesssim q^{\frac{d-1}{2}} \left(\Lambda_{k-2}(E)\right)^{\frac{1}{2}} \left(\Lambda_{k}(E)\right)^{\frac{1}{2}} .$$

Putting together all estimates above,  it follows that  we can choose a uniform constant $C>0$ such that
$$ \Lambda_{k}(E) \le \frac{C|E|^{k-1}}{q} +C q^{\frac{d-1}{2}} \left(\Lambda_{k-2}(E)\right)^{\frac{1}{2}} \left(\Lambda_{k}(E)\right)^{\frac{1}{2}} .$$
Solving this inequality for $\left(\Lambda_{k}(E)\right)^{\frac{1}{2}}$ implies that
$$ \sqrt{\Lambda_{k}(E)} \le \frac{Cq^{\frac{d+1}{2}}  \sqrt{\Lambda_{k-2}(E)}
+ \sqrt{ C^2 q^{d+1} \Lambda_{k-2}(E) + 4C q |E|^{k-1} }}{2q}. $$
Since $(A+B)^\alpha \approx A^\alpha + B^\alpha, $ the conclusion of our lemma follows from the above inequality.
\end{proof}

Inductively applying Lemma \ref{mlem}, we obtain the following result.
\begin{lemma} \label{mlemm}
Let $E$ be a subset of a regular variety $V\subset \mathbb F_q^d.$
If $k\ge 4$ is even, then
\begin{equation}\label{In1} \Lambda_k(E) \lesssim q^{\frac{ (d-1)(k-2)}{2}} \Lambda_2(E) + q^{-1} |E|^{k-1} \sum_{j=0}^{(k-4)/2}q^{(d-1) j} |E|^{-2j} .\end{equation}
On the other hand, if $k\ge 6$ is even, then 
\begin{equation}\label{In2} \Lambda_k(E) \lesssim q^{\frac{ (d-1)(k-4)}{2}} \Lambda_4(E) + q^{-1} |E|^{k-1} \sum_{j=0}^{(k-6)/2}q^{(d-1) j} |E|^{-2j} .\end{equation}
\end{lemma}

\begin{proof}  We prove (\ref{In1}) by an induction argument.
When $k=4$,   (\ref{In1})  clearly holds from Lemma \ref{mlem}.
Assume that the statement holds for an even integer $k\ge 4$. Namely, we assume that
\begin{equation}\label{A1} \Lambda_k(E) \lesssim q^{\frac{ (d-1)(k-2)}{2}} \Lambda_2(E) + q^{-1} |E|^{k-1} \sum_{j=0}^{(k-4)/2}q^{(d-1) j} |E|^{-2j}.
\end{equation}
Then it suffices to prove that 
$$ \Lambda_{k+2}(E) \lesssim q^{\frac{ (d-1)k}{2}} \Lambda_2(E) + q^{-1} |E|^{k+1} \sum_{j=0}^{(k-2)/2}q^{(d-1) j} |E|^{-2j}.$$
This shall follow from Lemma \ref{mlem} and the assumption (\ref{A1}). More precisely, we have
\begin{align*} \Lambda_{k+2}(E)&\lesssim q^{d-1} \Lambda_{k}(E) +q^{-1} |E|^{k+1}\\
 &\lesssim q^{d-1} \left( q^{\frac{ (d-1)(k-2)}{2}} \Lambda_2(E) + q^{-1} |E|^{k-1} \sum_{j=0}^{(k-4)/2}q^{(d-1) j} |E|^{-2j}   \right) + q^{-1} |E|^{k+1}\\
 &=q^{\frac{ (d-1)k}{2}} \Lambda_2(E) +  q^{d-2}|E|^{k-1} \sum_{j=0}^{(k-4)/2}q^{(d-1) j} |E|^{-2j}    + q^{-1} |E|^{k+1}\\
 &= q^{\frac{ (d-1)k}{2}} \Lambda_2(E) +  q^{d-2}|E|^{k-1}\sum_{j=1}^{(k-2)/2}q^{(d-1) (j-1)} |E|^{-2(j-1)}    + q^{-1} |E|^{k+1}\\
 &=q^{\frac{ (d-1)k}{2}} \Lambda_2(E) + q^{-1} |E|^{k+1} \sum_{j=1}^{(k-2)/2}q^{(d-1) j} |E|^{-2j}+  q^{-1} |E|^{k+1}\\
 &=q^{\frac{ (d-1)k}{2}} \Lambda_2(E) +q^{-1} |E|^{k+1} \sum_{j=0}^{(k-2)/2}q^{(d-1) j} |E|^{-2j}.
\end{align*}

The same argument  also yields (\ref{In2}). We leave the detail to the readers.
\end{proof}

The following corollary of Lemma \ref{mlemm} will be directly used in proving our main results.
\begin{corollary} \label{cor4.3}
Let $E$ be a subset of a regular variety $V\subset \mathbb F_q^d.$
In addition, assume that $|E| > q^{(d-1)/2}.$\\
$(1)$ If $k\ge 2$ is even, then
$$ \Lambda_k(E) \lesssim q^{\frac{ (d-1)(k-2)}{2}} |E| + q^{-1} |E|^{k-1}.$$
\noindent $(1^*)$ If $k\ge 3$ is odd, then
$$ \Lambda_{k-1}(E) \Lambda_{k+1}(E) \lesssim  q^{(d-1)(k-2)} |E|^2 + q^{\frac{(d-1)(k-3)-2}{2}} |E|^{k+1} +  q^{-2} |E|^{2k-2}.$$
\noindent $(2)$ If $k\ge 6$ is even, then 
$$ \Lambda_k(E) \lesssim q^{\frac{ (d-1)(k-4)}{2}} \Lambda_4(E) + q^{-1} |E|^{k-1} .$$
\noindent $(2^*)$ If $k\ge 7$ is odd, then
$$ \Lambda_{k-1}(E) \Lambda_{k+1}(E) \lesssim  q^{(d-1)(k-4)} \Lambda_4^2(E) + q^{\frac{(d-1)(k-5)-2}{2}} \Lambda_4(E) |E|^k  + q^{-2} |E|^{2k-2}.$$
\end{corollary}

\begin{proof} If $k=2$, then the statement $(1)$ is trivial, because $\Lambda_2(E)=|E|.$ Thus, to prove $(1)$, we may assume that $k\ge 4$ is even.  Since $\Lambda_2(E)=|E|,$ and $q^{(d-1)} |E|^{-2} <1$ by the hypothesis,  the statement $(1)$ follows immediately from the first part of Lemma \ref{mlemm}.
To prove $(1^*)$, notice that $(1)$ implies that for odd $k\ge 3,$ 
\begin{align*}&\Lambda_{k-1}(E)~ \Lambda_{k+1}(E) \\
\lesssim& \left(q^{\frac{ (d-1)(k-3)}{2}} |E| + q^{-1} |E|^{k-2}\right)\left(q^{\frac{ (d-1)(k-1)}{2}} |E| + q^{-1} |E|^{k}\right)\\
= &q^{(d-1)(k-2)} |E|^2 + q^{\frac{(d-1)(k-3)-2}{2}} |E|^{k+1} + q^{\frac{(d-1)(k-1)-2}{2}} |E|^{k-1} + q^{-2} |E|^{2k-2}.\end{align*}
Then the statement $(1^*)$ follows by observing that  if $|E|> q^{(d-1)/2},$ the second term above dominates the third term. Since we have assumed $|E|> q^{(d-1)/2}$, the statement $(2)$ is a direct consequence of the second part of  Lemma \ref{mlemm}. 
Finally, to prove $(2^*)$, notice from $(2)$ that for odd $k\ge 7,$ we have
\begin{align*}\Lambda_{k-1}(E)~ \Lambda_{k+1}(E)  & \lesssim q^{(d-1)(k-4)} \Lambda_4^2(E) + q^{\frac{(d-1)(k-5)-2}{2}} \Lambda_4(E) |E|^k  +
\\
& \hskip0,5in+ q^{\frac{(d-1)(k-3)-2}{2}} \Lambda_4(E) |E|^{k-2} + q^{-2} |E|^{2k-2}.\end{align*}
When $|E|> q^{(d-1)/2}$, it is easy to see that the second term is greater than the third term. Hence, the proof of the statement $(2^*)$ is complete.
\end{proof}

\section{Proofs of main results (Theorems \ref{T1} and \ref{T2})}
The proofs of main theorems will be complete by a direct application of Lemma \ref{keylem1} with Corollary \ref{cor4.3}.
Some routine algebra will be also needed for deriving the exact results.

\subsection{Proof of Theorem \ref{T1} (on regular varieties)}
We restate and prove Theorem \ref{T1}. \\

\noindent {\bf Theorem \ref{T1}.} Suppose that  $V\subset \mathbb F_q^d$ is a regular variety.
In addition, assume that $k\ge 3$ is an integer and $E\subset V.$
Then if $|E|\ge C q^{\frac{d-1}{2} +\frac{1}{k-1}}$ for a sufficiently large $C>0$,  we have 
 $$\Delta_k(E) \supset \mathbb F_q^* \quad \mbox{for even}~~d\ge 2,$$
and 
$$\Delta_k(E) =\mathbb F_q \quad \mbox{for odd}~~d\ge 3.$$
\begin{proof} {\bf Case 1.} Assume that $k\ge 4$ is an even integer. By $(1)$ of Lemma \ref{keylem1} and $(1)$ of Corollary \ref{cor4.3}, it suffices to prove that
if $E\subset V$ with $|E|\ge C q^{\frac{d-1}{2} +\frac{1}{k-1}},$ then 
$$ |E|^k\gtrsim q^{\frac{d+1}{2}} (q^{\frac{ (d-1)(k-2)}{2}} |E| + q^{-1} |E|^{k-1}).$$
By a direct comparison, this inequality follows.\\

{\bf Case 2.} Assume that $k\ge 3 $ is an odd integer. By $(2)$ of Lemma \ref{keylem1}, it is enough to prove that if $E\subset V$ with $|E|\ge C q^{\frac{d-1}{2} +\frac{1}{k-1}},$ then
 $$|E|^{2k} \gtrsim q^{d+1}  \left( \Lambda_{k-1}(E) \Lambda_{k+1}(E) \right).$$
Invoking $(1^*)$ of Corollary \ref{cor4.3},  we only need to prove that if $|E|\ge C q^{\frac{d-1}{2} +\frac{1}{k-1}}$, then
$$|E|^{2k} \gtrsim q^{d+1} \left(q^{(d-1)(k-2)} |E|^2 + q^{\frac{(d-1)(k-3)-2}{2}} |E|^{k+1} +  q^{-2} |E|^{2k-2}\right).$$
Since $|E|>q^{\frac{d-1}{2} +\frac{1}{k-1}}$, this inequality is simply proved by comparing $|E|^{2k}$ with each term in the right hand side.
\end{proof}

\subsection{Proof of Theorem \ref{T2} (on nondegenerate regular curve)}
To complete the proof of Theorem \ref{T2}, we shall invoke the known results on the extension problem for a nondegenerate regular curve. we begin by reviewing  the extension problems for varieties of $\mathbb F_q^d.$ Let $V\subset \mathbb F_q^d, d\ge 2,$ be an algebraic variety. Denote by $d\sigma$ the normalized surface measure on $V$, which is defined by the relation
$$ \int f(x)~d\sigma(x) =\frac{1}{|V|} \sum_{x\in V} f(x) \quad\mbox{for}~~f:\mathbb F_q^d \to \mathbb C.$$
For  a function $f:\mathbb F_q^d \to \mathbb C$ and the normalized surface measure $d\sigma$ on $V$,  the inverse Fourier transform of the measure $f d\sigma$ is defined as
$$ (fd\sigma)^\vee(m):= \int \chi(m\cdot x)~f(x)~d\sigma(x) =\frac{1}{|V|} \sum_{x\in V} \chi(m\cdot x)~f(x),$$
where $m$ is an element of $\mathbb F_q^d$ with the counting measure $dm.$ Then by the usual definition of norms,  we can write that for $1\le p,r < \infty$,
$$ \|(fd\sigma)^\vee\|_{L^r(\mathbb F_q^d, dm)} := \left( \sum_{m\in \mathbb F_q^d} \left| (fd\sigma)^\vee(m)\right|^r\right)^{\frac{1}{r}},  $$
and
$$ \|f\|_{L^p(V, d\sigma)} := \left( \frac{1}{|V|} \sum_{x\in V} |f(x)|^p \right)^{\frac{1}{p}}.$$
With the above notation, we say that the $L^p\rightarrow L^r$ extension estimate for $V$ holds if there is a constant $C>0$ independent of $q$, the size of the underlying finite field $\mathbb F_q,$ such that
$$ \|(fd\sigma)^\vee\|_{L^r(\mathbb F_q^d, dm)} \le C \|f\|_{L^p(V, d\sigma)} \quad \mbox{for all functions}~~f:\mathbb F_q^d \to \mathbb C.$$

When  sets $E\subset \mathbb F_q^2$ are contained in a nondegenerate regular curve $V$, a sharp upper bound of $\Lambda_4(E)$ can be obtained by using  the $L^2\rightarrow L^4$ extension estimate for the variety $V,$ a consequence of Theorem 1.1 in \cite{KS10}. 

\begin{lemma}\label{lem5.1} Let $V\subset \mathbb F_q^2$ be a nondegenerate regular curve. Then there is a constant $C>0$ independent of $q$ such that
$$ \Lambda_4(E) \le C |E|^2 \quad \mbox{for all}~~ E\subset V.$$  \end{lemma}
\begin{proof} Theorem 1.1 in \cite{KS10} implies that
$$ \|(fd\sigma)^\vee\|_{L^4(\mathbb F_q^2, dm)} \lesssim \|f\|_{L^2(V, d\sigma)} \quad \mbox{for all functions}~~f:\mathbb F_q^2 \to \mathbb C.$$
Taking $f$ as the characteristic function on the set $E\subset V$,  the expansion of norms yields
$$ \sum_{m\in \mathbb F_q^2} \left|\sum_{x\in E} \chi(m\cdot x)\right|^4 \lesssim q^2 |E|^2,$$
where we also used the fact that $|V|\sim q.$
Write 
$$\sum_{m\in \mathbb F_q^2} \left|\sum_{x\in E} \chi(m\cdot x)\right|^4 =\sum_{m\in \mathbb F_q^2} \sum_{x,y,z,w\in E} \chi(m\cdot(x+y-z-w))$$
and use the orthogonality relation of $\chi.$ Then the statement of the lemma follows.
\end{proof}

Now, we restate and prove Theorem \ref{T2}.\\

\noindent {\bf Theorem \ref{T2}.} Suppose that  $E$ is contained in a nondegenerate regular variety $V\subset \mathbb F_q^2.$  Then if $k\ge 4$ is an integer and  $|E|\ge C q^{\frac{1}{2}+\frac{1}{2k-4}}$ for a sufficiently large constant $C>0$, we have $\mathbb F_q^* \subset \Delta_k(E).$
\begin{proof} {\bf Case 1.}  Assume that $k\ge 4$ is an even integer. If $k=4$, then by $(1)$ of Lemma \ref{keylem1}, it suffice to prove that 
if $|E|\ge C q^{3/4},$ then $|E|^4 \gtrsim q^{3/2} \Lambda_4(E).$ 
Since $\Lambda_4(E) \lesssim |E|^2$ by Lemma \ref{lem5.1}, this clearly holds.
Next, let us assume that $k\ge 6$ is an even integer. Combining $(1)$ of Lemma \ref{keylem1} with $(2)$ of Corollary \ref{cor4.3},  it will be enough to show that
if $|E|\ge C q^{\frac{1}{2}+\frac{1}{2k-4}}$, then 
$$ |E|^k \gtrsim q^{\frac{3}{2}} \left( q^{\frac{k-4}{2}} \Lambda_4(E) + q^{-1} |E|^{k-1}  \right).$$
From the conclusion of Lemma \ref{lem5.1}, that is $\Lambda_4(E) \lesssim |E|^2,$  it suffices to show that
if  $|E|\ge C q^{\frac{1}{2}+\frac{1}{2k-4}},$ then 
$$  |E|^k \gtrsim q^{\frac{3}{2}} \left( q^{\frac{ k-4}{2}} |E|^2 + q^{-1} |E|^{k-1}  \right) =q^{\frac{k-1}{2}} |E|^2 + q^{\frac{1}{2}} |E|^{k-1}.$$
This inequality holds because  $|E|^k $ dominates both $q^{\frac{k-1}{2}} |E|^2$ and $q^{\frac{1}{2}} |E|^{k-1}$
provided that $|E|\ge C q^{\frac{1}{2}+\frac{1}{2k-4}}.$ Thus, we have completed the proof of Theorem \ref{T2} for even integers $k\ge 4.$\\

\noindent {\bf Case 2.} Assume that $k\ge 5$ is an odd integer.
In this case,  applying $(2)$ of Lemma \ref{keylem1} with $d=2,$ it will be enough to show that for every $E\subset V \subset \mathbb F_q^2$ with $|E|\ge C q^{\frac{1}{2}+\frac{1}{2k-4}},$ we have
$$|E|^k \gtrsim q^{3/2}  \left( \Lambda_{k-1}(E) \Lambda_{k+1}(E) \right)^{1/2}$$
or
\begin{equation}\label{eq5.1} |E|^{2k} \gtrsim q^3 \left( \Lambda_{k-1}(E) \Lambda_{k+1}(E) \right).\end{equation}
First, let us prove this for $k=5.$ 
We must prove that if $|E|\ge C q^{2/3}$, then $|E|^{10}\gtrsim q^{3}\left( \Lambda_4(E) \Lambda_{6}(E) \right). $
Since $\Lambda_4(E)\lesssim |E|^2$ by Lemma \ref{lem5.1}, we see from $(2)$ of Corollary \ref{cor4.3} that
$$\Lambda_6(E) \lesssim q |E|^2 + q^{-1} |E|^5.$$
Thus, \eqref{eq5.1} will hold for $k=5$ if we prove that for every $E\subset V$ with $|E|\ge C q^{2/3},$
$$ |E|^{10} \gtrsim q^3 |E|^2 ( q |E|^2 + q^{-1} |E|^5) = q^{4} |E|^4 + q^2 |E|^7.$$
By a direct comparison, this is clearly true and we complete the proof for $k=5.$
Finally, we assume that $k\ge 7$ is an odd integer and prove that if $|E|\ge C q^{\frac{1}{2}+\frac{1}{2k-4}},$ then \eqref{eq5.1} holds. Using $(2^*)$ of Corollary \ref{cor4.3} with $d=2$, it suffices to prove that if $|E|\ge C q^{\frac{1}{2}+\frac{1}{2k-4}},$ then 
$$|E|^{2k} \gtrsim q^3 \left(q^{k-4} \Lambda_4^2(E) + q^{\frac{k-7}{2}} \Lambda_4(E) |E|^k  + q^{-2} |E|^{2k-2}\right). $$
Since $\Lambda(E) \lesssim |E|^2$ by Lemma \ref{lem5.1},  we only need to prove that
if $|E|\ge C q^{\frac{1}{2}+\frac{1}{2k-4}},$ then
$$ |E|^{2k} \gtrsim q^3 \left(q^{k-4} |E|^4 + q^{\frac{k-7}{2}}  |E|^{k+2}  + q^{-2} |E|^{2k-2}\right). $$
This statement is obvious by a direct  calculation, and this completes the proof for odd $k\ge 7.$
\end{proof}
 





\begin{thebibliography}{7}

\bibitem{BKT04} J.~Bourgain, N.~ Katz, and T.~ Tao,  \emph{A sum-product estimate in finite fields, and applications}, Geom. Funct. Anal. 14 (2004), 27--57.

\bibitem{BHIPR14} J. Bennett, D. Hart, A. Iosevich, J. Pakianathan, M. Rudnev, \emph{Group actions and Geometric combinatorics in $\mathbb{F}_q^d$}, arxiv:1311.4788v1.

\bibitem{CEHIK09} J.~ Chapman, M.~ Erdo\~{g}an, D.~ Hart, A.~ Iosevich, and D.~ Koh, \emph{Pinned distance sets, Wolff's exponent in finite fields and sum-product estimates}, Math.Z., \textbf{271}, (2012),
63--93

\bibitem{CKP14} D. Covert, D. Koh, and Y. Pi, \emph{On the sums of any $k$ points in finite fields},  arXiv:1403.6138.

\bibitem{Er05} M. Erdo\~{g}an, \emph{ A bilinear Fourier extension theorem and applications to the distance set problem,} Internat. Math. Res. Notices {\bf 23} (2005), 1411--1425.

\bibitem{Er46} P.~Erd\H os, \emph{ On sets of distances of $n$ points}, Amer. Math. Monthly 53, (1946), 248--250.

\bibitem{Fa85} K.~ Falconer, \emph{ On the Hausdorff dimension of distance sets,} Mathematika, 32 (1985), 206--212.

\bibitem{GuKa10}  L. Guth and N. Katz, \emph {On the Erd\"os distinct distance problem in the plane}, Annals of Mathematics, Volume 181 (2015), pages 155-190.

\bibitem{HIKR10} D.~ Hart, A.~ Iosevich, D.~ Koh and M.~ Rudnev, \emph{ Averages over hyperplanes, sum-product theory in vector spaces over finite fields and the Erd\"os-Falconer distance conjecture}, Trans. Amer. Math. Soc. Volume 363, Number 6,  (2011),  3255–3275.

\bibitem{IR07} A.~ Iosevich and M. ~Rudnev, \emph{Erd\"{o}s distance problem in vector spaces over finite fields}, Trans. Amer. Math. Soc. 359 (2007), 6127--6142.

\bibitem{KS10} D. ~Koh and C.~ Shen, \emph {Sharp extension theorems and Falconer distance problems for algebraic curves in two dimensional vector spaces over finite fields }, Revista Matematica  Iberoamericana, Vol   28  Issue 1, 2012 159-180.

\bibitem{KS12} D. Koh and H. Sun, \emph{Distance sets of two subsets of   vector spaces over  finite fields}, PAMS, to appear,  arXiv:1212.5305.

\bibitem{LN97} R. Lidl and H. Niederreiter, \emph{ Finite fields,} Cambridge University Press, (1997).

\bibitem{MT04} G. Mockenhaupt, and T. Tao, \emph{Restriction and Kakeya phenomena for finite fields}, Duke Math. J. {\bf 121} (2004), no. 1, 35--74.

\bibitem{SV08} J.~ Solymosi and V.~ Vu, \emph{Near  optimal bounds for the number of distinct distances in high dimensions}, Combinatorica, Vol 28, no 1 (2008), 113--125.

\bibitem{Wo99} T. Wolff, \emph{Recent work connected with the Kakeya problem}, Prospects in mathematics (Princeton, NJ, 1996),  129--162, Amer. Math. Soc., Providence, RI, 1999.

\end{thebibliography}
\end{document}